\documentclass[a4paper,10pt]{amsart}

\usepackage{xcolor}

\usepackage{amsmath,latexsym,amssymb}
\usepackage[german,english]{babel}
\usepackage[applemac]{inputenc}
\usepackage{float}
\usepackage{graphicx}
\usepackage{lineno}

\newtheorem{prop}{Proposition}[section]
\newtheorem{thm}[prop]{Theorem}
\newtheorem{lemma}[prop]{Lemma}
\newtheorem{cor}[prop]{Corollary}

\newtheorem{conj}[prop]{Conjecture}

\newenvironment{defi}
{\stepcounter{prop}{\noindent \bf Definition
\thesection.\arabic{prop}.}
}

\newenvironment{remark}
{\stepcounter{prop}{\noindent \bf Remark
\thesection.\arabic{prop}.}
}

\renewcommand{\proof}[1][]{{\it Proof#1: }}
\renewcommand{\qed}[1][3mm]{\hspace*{\fill} $\Box$ \vspace{#1}}

\newcommand{\ZZ}{{\mathbb Z}}
\newcommand{\CC}{{\mathbb C}}

\newcommand{\PP}{{\mathbf P}}
\newcommand{\FF}{{\mathbf F}}

\renewcommand{\varpi}{\wp}

\newcommand{\tto}{\longrightarrow}

\newcommand{\dfami}{{ \mathcal D}}

\newcommand{\hfami}{{ \mathcal H}}

\newcommand{\lfami}{{ \mathcal L}}
\newcommand{\mfami}{{ \mathcal M}}

\newcommand{\ofami}{{ \mathcal O}}

\newcommand{\cutoff}[1]{}

\newcommand{\labell}[1]{\label{#1}}

\newcommand{\rk}{\operatorname{rk}}

\renewcommand{\span}{\operatorname{span}}

\newcommand{\hsix}{\hfami^{ev}_4(6)}
\newcommand{\piorb}{\pi_1^{orb}}

\begin{document}

\title
{\noindent\center  \bf $\pi_1$ of trigonal loci of strata of\\
abelian differentials\\
}

\author[M: L\"onne]{Michael L\"onne}
\address{Mathematisches Institut der Universit\"at Bayreuth, Universit\"atsstr.\ 30, 95447 Bayreuth, Germany}
\email{michael.loenne@uni-bayreuth.de}

\begin{abstract}
We investigate locally closed subspaces of projectivized strata of abelian differentials
which classify trigonal curves with canonical divisor a multiple of a trigonal divisor.
We describe their orbifold structure using linear systems on Segre-Hirzebruch surfaces
and obtain results for their orbifold fundamental groups.

Most notable among these orbifolds is the connected component 
$\PP\hsix$, the projectivisation
of the space $\hsix$ of abelian differentials on non-hyperelliptic
genus $4$ curves with a single zero of multiplicity 6
providing an even spin structure. 
Its orbifold fundamental group is identified with the quotient of the Artin group
of type $E_8$ by its maximal central subgroup.
\footnotemark
\end{abstract}

\maketitle

\setcounter{footnote}{1}

\footnotetext{
This result was previously obtained by Giannini \cite{gia}, and we give an independent proof
using an alternative argument.
}

\markboth{loci of trigonal abelian differentials, \today}
{loci of trigonal abelian differentials, \today}

\pagestyle{myheadings}

%
%
%
%
%
%
%
%
%
%
%
%
%
%
%
%
%
%
%

\begin{center}
\it
dedicated to the memory of Wolfgang Ebeling 
\end{center}

\section{Introduction}

Wolfgang Ebeling was the advisor of my thesis and following that, he provided
constant support and inspiration during the preparation of my habilitation thesis. 
His way to study and to teach singularity theory had an essential impact on the choice
of its topic which aimed at a better understanding of fundamental
groups of discriminant complements in versal unfoldings of hypersurface singularities.

These groups have since been a recurring object in the study of various spaces of curves
and in the present paper we want to use our approach
to the study of moduli spaces $\hfami_g$ of 
abelian differentials, particularly -- regarding the $\hfami_g$ as moduli stacks --
to the study of their orbifold structure.
The $\hfami_g$ classify pairs $C,\varphi$ consisting of a complex algebraic
curve $C$ of genus $g$ and a non-zero section $\varphi\in H^0(C,\omega_C)$
of the canonical bundle $\omega_C$ over $C$.

The associated projectivized moduli spaces $\PP\hfami_g$ classify pairs
$C,D$, where $D$ is an effective divisor of degree $2g-2$ on $C$ given as the 
zero divisor of a differential $\varphi$.

Both spaces decompose naturally into strata $\hfami_g(k_1,...,k_r)$ resp.\ 
$\PP\hfami_g(k_1,...,k_r)$ by the multiplicities of the zeroes of $\varphi$ resp.\
those of the points of $D$. 
It is the celebrated achievement of Kontsevich and Zorich \cite{koz} to find and
characterize the connected components of these strata.

Since then it is an intriguing problem how they relate to the moduli space of curves
$\mfami_g$ with Kontsevich and Zorich putting a special emphasis on the following
two question which they conjectured to have positive answers:
\begin{quotation}
Are the strata orbifold quotients of contractible spaces?
\\[2mm]
Do their fundamental groups relate nicely with mapping class groups? 
\end{quotation}

Addressing the second question Calderon Salter \cite{cs}
successfully described the image of
the monodromy induced by the forgetful map $\hfami_g(k_1,...,k_r)\to \mfami_g$
using techniques primarily from geometric group theory.

A more algebraic geometric approach is taken by Looijenga and Mondello \cite{lm}.
They describe the orbifold fundamental groups for strata of genus $3$ and settle
the first question in the affirmative for most of these strata,
but they express their doubt, whether their orbifold fundamental groups 
can be shown to be commensurable with mapping class groups.
\medskip

We propose the study of a new kind of loci defined in analogy with the purely
hyperelliptic strata $\hfami^{hyp}(2g-2)$ and $\hfami^{hyp}(g-1,g-1)$, which are 
exceptional components of $\hfami_g(2g-2)$ and $\hfami_g(g-1,g-1)$ respectively.
Recall from \cite[Rem. 3]{koz}:
\begin{quote}
Points of $\hfami^{hyp}(2g-2)$ respectively of $\hfami^{hyp}(g-1,g-1)$
are abelian differentials on hyperelliptic curves of genus $g$ which have 
a single zero of multiplicity $2g-2$ invariant under the hyperelliptic involution
respectively a pair of zeroes of orders $g-1$ symmetric to each other with
respect to the hyperelliptic involution.
\end{quote}
Pairs $(C,D)$ in $\PP\hfami^{hyp}(2g-2)$ respectively in $\PP\hfami^{hyp}(g-1,g-1)$
can thus be characterized by the property that $D$ is a multiple of a divisor in the
${g_2}^1$ of $C$ having support in a single point, respectively in a pair of distinct
point.
\medskip

For our generalization we increase the gonality und consider trigonal curves with
canonical divisors a multiple of a divisor in the trigonal linear system.
The precise definition reads as follows:
\medskip

\begin{defi}
\label{def_strict}
A pair $(C,D)\in \PP\hfami_g$ is called \emph{strictly trigonal} if
\begin{enumerate}
\item
$C$ is strictly trigonal, i.e.\ $C$ is not hyperelliptic with a trigonal linear system $g_3^1$,
\item
$D$ is an integral multiple of a divisor $L$ in a trigonal linear system $g_3^1$ of $C$.
\end{enumerate}
The subspace $\PP \hfami^{tri}_{g}$ of $\PP \hfami_g$ of strictly trigonal pairs
is the moduli space of these pairs. 
\end{defi}
\bigskip

We will see below that these spaces are non-empty only if $g=3k+1$ for some $k\geq1$
and only intersect strata of types $(6k)$, $(4k,2k)$ or $(2k,2k,2k)$ with spin structure of the 
parity of $g$. The number of zeroes is $1$, $2$ and $3$ respectively for the differentials.

In these cases we can describe the orbifold structure in very concrete terms
using discriminant complements of suitable linear systems on Hirzebruch surfaces
(later on, we will be more precise and explain all ingredients):

\begin{thm}
\label{orbi}
Loci of strictly trigonal abelian differentials in $\PP \hfami_g$ of genus $g=3k+1$, 
$k\geq1$ can be identified in the orbifold sense as
\[
\begin{array}{ccc}
\PP \hfami^{tri}_{3k+1}(6k) & \cong & (\lfami^k_1 -\dfami_1 ) / \CC^* \\[2mm]
\PP \hfami^{tri}_{3k+1}(4k,2k) & \cong & \lfami^k_{2} -\dfami_2  \\[2mm]
\PP \hfami^{tri}_{3k+1}(2k,2k,2k) & \cong & 
(\lfami^k_{3} -\dfami_3 -\dfami' ) / \CC\rtimes (\CC^*)^2 
\end{array}
\]
where $\lfami^k_i$ are linear subsystems of $|3\sigma_0|$ 
on the Hirzebruch surface $\FF_{k+1}$, the
$\dfami_i$ are the respective discriminant divisors corresponding to singular curves,
and $\dfami'$ is the divisor corresponding to abelian differentials with less than $3$
zeroes.
\end{thm}

The same kind of description was obtained in \cite{bl} for the locus of trigonal curves
of maximal Maroni invariant and its orbifold fundamental group -- in fact equal to that of
$\PP \hfami^{tri}_{3k+1}$ -- was determined using unfoldings of isolated plane curve
singularities.

Similarly here, for types $(6k)$ and $(2k,2k,2k)$
we are going to describe the topology of the discriminant 
complements in the theorem with the help of discriminant complements in 
unfoldings of isolated plane curve singularities.
Their fundamental groups, studied in \cite{intern} and \cite{bries} under
the name of \emph{discriminant knot groups}, can be used to express
the orbifold fundamental groups in the present setting:

\begin{cor}
\label{first}
The orbifold fundamental group of $\PP \hfami^{tri}_{3k+1}(6k)$ fits
into a short exact sequence 
\[
1 \quad \tto \quad \ZZ \quad \tto \quad \pi^K(y^3+x^{3k+2})
\quad \tto \quad
\piorb \PP \hfami^{tri}_{3k+1}(6k)
\quad \tto \quad 1
\]
as quotient by a central subgroup of the discriminant knot group 
$\pi^K$ of the singular plane curve germ 
$y^3+x^{3k+2}$.
\end{cor}

\begin{cor}
\label{third}
The orbifold fundamental group of $\PP \hfami^{tri}_{3k+1}(2k,2k,2k)$ fits
into a short exact sequence 
\[
1 \; \to \; \ZZ^2 \; \to \; 
\pi^K(y^3+x^{3k+3}) \rtimes \pi^K(y^3+x^2)
\; \to \;
\piorb \PP \hfami^{tri}_{3k+1}(2k,2k,2k)
\; \to \; 1
\]
as quotient of a semi-direct product of discriminant knot groups $\pi^K$ 
of singular plane curve germs
$y^3+x^{3k+3}$ and $y^3+x^2$.
\end{cor}

Another aspect we want to emphasize arrises in the special case $k=1, g=4$
of our results. From the basic observation that all curves of genus $4$ are trigonal
we can sharpen the theorem 
to obtain an alternative proof of \cite[Thm.1]{gia}
(this is not possible for $k>1$ by a simple dimension
count).

\begin{cor}
\label{pezzo}
The projectivization $\PP \hsix$
of the space $\hsix$ of abelian differentials of even spin structure
is isomorphic to a quotient 
\[
\PP \hsix
\quad \cong \quad
(\CC^8 - \dfami_{E_8})/\CC^*
\]
of the unfolding of the simple $E_8$ singularity and
the orbifold fundamental group is isomorphic to a quotient
of the Artin group of type $E_8$ by its centre:
\[
\piorb \PP \hsix
\quad \cong \quad
Ar(E_8)/Centre
\]
\end{cor}

This result is very similar to results of \cite{lm}.
Indeed it should be interesting to look at $\PP \hsix$
from the point of view of anti-canonical divisors on a degree one
delPezzo surface.
\medskip

Conspicuously we make no claim in case of type $(4k,2k)$
where obviously $\piorb = \pi_1$.
However, the discriminant complements $\lfami^k_{2} -\dfami$ is not 
induced from a versal unfolding of plane curve germs and
an identification with a discriminant knot of a singular plane curve germs
is not possible.

Let us sketch though an idea how to get around that obstacle:
As finitely presented groups, the groups above should be
studied as instances of \emph{secondary braid groups} ${}^2Br$
associated to positive braid words, as proposed by the author
and investigated together with Baader \cite{bal}.
Indeed, there is the following identification with
secondary braid groups
associated to positive words in the standard generators $\sigma_1,\sigma_2$ 
of the braid group $Br_3$:
\[
\pi^K(y^3+x^{3k+2})
\quad \cong \quad
{}^2Br\big( (\sigma_1\sigma_2)^{3k+2} \big)
\]
\[
\pi^K(y^3+x^{3k+3}) \quad \cong \quad
{}^2Br\big( (\sigma_1\sigma_2)^{3k+3} \big)
\]
\[
\pi^K(y^3) \quad \cong \quad
{}^2Br\big( (\sigma_1\sigma_2)^{2} \big)
\quad \cong \quad
Br_3
\]
These positive words define braids which close to the links in $S^3$
of the respective plane curve singularities.
Moreover it is true that the link at infinity of the smooth curves in $\lfami^k_1,
\lfami^k_{2}$ and $\lfami^k_{3}$
is equal to the closure of the
respective braids 
\[
(\sigma_1\sigma_2)^{3k+2},\quad
\sigma_1(\sigma_2\sigma_1)^{3k+2} =
(\sigma_1\sigma_2)^{3k+2}\sigma_1 \quad\text{and}\quad
(\sigma_1\sigma_2)^{3k+3}
\]
So by interpolation and after some encouraging first investigations we are confident 
to propose the following conjecture:

\begin{conj}
\label{two}
The orbifold fundamental group of $\PP \hfami^{tri}_{3k+1}(4k,2k)$ is given by
an isomorphism
\[
\piorb \PP \hfami^{tri}_{3k+1}(4k,2k)
\quad \cong \quad
{}^2Br\big( \sigma_1(\sigma_2\sigma_1)^{3k+2} \big).
\]
where the secondary braid group has a finite presentation by
\[
\left\langle t_1, \dots, t_{6k+3} \:\Bigg|\:
\begin{array}{ll}
t_i t_j = t_j t_i & \text{if } |i-j|>2 \\
t_i t_j t_i = t_j t_i t_j & \text{if } |i-j|\leq2 \\
t_i t_{i+1}t_{i+2}t_i = t_{i+1}t_{i+2}t_i t_{i+1}
\end{array}
\right\rangle_.
\]
\end{conj}
\medskip

While this has to be deferred to a future paper, we are going to proof our
stated results pursuing the following steps.
First we are going to review and use canonical geometry of trigonal curves. 
Then we reduce the canonical global quotient construction of our loci to
a quotient constructions on a vector space of polynomials.

The two subsequent section provide a set of transformations and employ
them to accomplish the proof of the theorem.
In the final section we give the proofs of the corollaries.  

\section{canonical geometry}

There is a well-known approach to the study of non-hyperelliptic trigonal curves 
and their moduli. We refer to \cite{sv} which provides all the information necessary
for the following set-up.

Let $C$ be a non-hyperelliptic curve of genus $g\geq 4$, then it is canonically
embedded into projective space by its canonical linear system.
\[
\phi_{|\omega_C|}: C \to \PP^{g-1}:= \PP H^0(C, \omega_C).
\]
Every canonical divisor of $C$ is thus identified with a hyperplane section 
of the image $C_{can}:=\phi(C)$.

If moreover $C$ is trigonal, i.e.\ $C$ has a base point free $g_3^{\,1}$, then the
canonical image of $C$ is contained in a rational scroll $S_C$ projectively equivalent
to some
\[
S_{mn} \quad = \quad \Big\{ (z_0: \dots : z_{g-1}) \mid
\rk \Big(
\begin{array}{cccccc}
z_0 & \ldots & z_{n-1} & z_{n+1} & \ldots & z_{n+m} \\
z_1 & \ldots & z_{n} & z_{n+2} & \ldots & z_{n+m+1}
\end{array}
\Big) <2 \Big\}
\]
where $g=m+n+2$ and $m\leq n$.
The scroll $S_C$ for the curve $C$ is cut out by the quadrics containing
$C_{can}$ such that the divisors of the $g_3^{\,1}$ on $C$ and the 
lines on $S_C$ correspond bijectively. 

The $g_3^{\,1}$ on $C$ is unique except for
the case of $g=4$ and $n=m=1$ when there are two 
$g_3^{\,1}$ on $C$ and two corresponding rulings of $S_C$ by lines but that
case will not play a role in this article.
\medskip

The difference $e:=n-m$ is called the \emph{Maroni invariant} and known to take 
all integral values subject to
\[
0\leq n-m \leq \lfloor \frac{g+2}3 \rfloor, \quad n-m \equiv_2 g.
\]
The Maroni invariant determines the smooth model of $S_{mn}$ to be
the rational ruled Segre-Hirzebruch surface
\[
\FF_{n-m} \quad := \quad \PP\; \ofami(n) \oplus \ofami(m)
\]
The zero section of $\ofami(n)$ provides a moveable curve $\sigma$
on $\FF_{n-m}$,
the zero section of $\ofami(m)$ gives a curve $E$, which has self-intersection
number $-e\leq0$ and which is rigid for $e>0$. Together with the class of a fibre 
$L$ of the ruling, either section class generates the Picard group of $\FF_e$.  

$S_{nm}$ is then the image for a map associated with the complete
linear system
\[
\mid \sigma + m L \mid
\quad = \quad
\mid E + n L \mid 
\]
which defines an embedding except for the case $g=4, e=2$ when $m=0$ and
the exceptional section $E$ of $\FF_2$ is contracted. In that case
$S_{2,0}$ is the quadric cone in $\PP^3$. 
The image of $E$ is is any case given by a parametrization
\[
E: \quad 
\left\{
(0: \ldots : 0 : a^m : a^{m-1}b : \ldots : ab^{m-1} : b^m) \mid
(a:b) \in \PP^1
\right\},
\]
a twisted rational curve for $m>0$ and the point $(0: \ldots : 0 : 1)$ if $m=0$.

We are now sufficiently prepared to prove the following result,
that being part of a strictly trigonal pair, see Def.\ref{def_strict}.1, puts
additional conditions on the trigonal curve. 

\begin{prop}
Suppose $(C,D)$ is a strictly trigonal pair in $\PP\hfami_g$, then
\begin{enumerate}
\item
there is $k>0$ such that $g=3k+1$,
\item
$S_C$ is projectively equivalent to $S_{2k,k-1}$
( thus $C$ has maximal Maroni invariant).
\item
$D$ is cut out by a hyperplane $H$ with $H$ $\cap$ $S_{2k,k-1}$ 
projectively equivalent to
\[
\left\{
\begin{array}{cl}
2kL & \text{if  } k=1\\
2kL+E & \text{if  } k>1
\end{array}
\right.
\]
as a divisor, where $L$ is a line of the scroll.
\end{enumerate}
\end{prop}

\begin{proof}
Since $D$ is an integer multiple of a divisor $D' \in g_3^1(C)$, this is 
also true for their degrees, so $3$ divides $2g-2$ and thus divides $g-1$. Hence
$g=3k+1$ and $D=2k D'$. Since there are no strictly trigonal curves for $g=1$,
$k>0$ and part $i)$ is proved.

A canonical image of a trigonal $C$ lies on a rational scroll $S_{nm}$
with $e= n-m \leq k+1$ of the same parity as $g$ and $k+1$.
Via $\FF_e\to S_{nm}$ we consider $C$ as a curve on $\FF_e$,
its $g_3^1$ as the restriction of the linear system $\mid L\mid$ and
any hyperplane section of $S_{nm}$ as an effective divisor 
$H_S$ on $\FF_e$. Then
\[
C \,\in\; \mid 3\sigma_0+(2m-n+2) L \mid, \quad
H_S \,\in\; \mid \sigma_0+ mL \mid \;=\; \mid E+ nL \mid .
\]
Accordingly the hyperplane section $D$ of $C$ is given as $H_S\cap C$ on $\FF_e$
and the defining property of strictly trigonal
pairs poses a strong condition on $H_S$:
\begin{eqnarray}
\label{schnitt}
H_S \cap C \quad = \quad D &=& 2kD' \quad = \quad 2k L \cap C \quad \text{as divisors
of degree }2g-2=6k
\end{eqnarray}
\begin{eqnarray}
H_S \cap C \quad = \quad D &=& D' \quad = \quad L' \cap C \quad \text{as sets}
\end{eqnarray}
for a single fibre $L'$ of the ruling.
The class of $H_S$ forces this divisor to have the following irreducible decomposition
\[
H_S = \sigma_H + a L
\]
with some smooth section $\sigma_H$ to the ruling of $\FF_e$ and $a\geq0$. Indeed, $\sigma_H$
must be disjoint from $C$ which implies $a=n = 2k$ and thus the remaining
claims. 
Otherwise $\sigma_H \cap C=\sigma_H \cap L$ were a single point 
contributing with multiplicity at least $3$ to both $H_S\cap C$
and $L\cap C$ which is impossible since all three divisors are smooth
and $\sigma_H$ intersects $L$ transversally.
\qed
\end{proof}

Let us note, that on $S_{2k,k-1}\subset \PP^{g-1}$ the hyperplane $H_0$ given by $z_0=0$
cuts out the effective divisor
\[
H_S = H_0\cap S_{2k,k-1} = 
E \cup 2k L_0
\]
where $L_0$ is a line of the ruling and there is a map $\CC^2\to S_{2k,k-1}$
which is an isomorphism onto the complement of $E \cup L_0$:
\[
\begin{array}{rcl}
\CC^2 & \tto & S_{2k,k-1} \\
x,y & \mapsto & (1:x:\dots:x^n:x:xy:\dots:x^my)
\end{array}
\]
It induces an isomorphism
\[
\mid 3\sigma_0+(2m-n+2) L \mid
\quad \cong \quad
\PP V^k
\]
where $V^k\subset \CC[x,y]$ is the vector space of polynomials
$f= s y^3 + r(x) y^2 + p(x) y + q(x)$ spanned by monomials
\[
x^i y^j \quad \text{of weighted degree} \quad
 i + (k+1) j \; \leq \; 3k + 3
\]
that is, $s,r(x),p(x),q(x)\in\CC[x]$ are polynomials of degrees at most
$0,k+1,2k+2,3k+3$ respectively.
\medskip

In the end we want to discard the $f\in V^k$ in a discriminant divisor 
corresponding to singular curves, hence we define
\medskip

\begin{defi}
An element $f\in V^k$ is called \emph{regular} if it has the following equivalent
properties:
\begin{enumerate}
\item
$C_f\in \mid 3\sigma_0+(2m-n+2) L \mid$ corresponding to $f=0$ is a smooth curve on $\FF_{k+1}$, 
\item
$s\neq 0$, $f=0$ defines a smooth curve in $\CC^2$ and and $C_f$ has no singularity on $L_0$.
\end{enumerate}
\end{defi}

\section{orbifold structure}

Let us now turn to the moduli spaces of strictly trigonal pairs.
The orbifold structure for the open part $\PP \hfami_g^{non-hyp}$ of $\PP \hfami_g$ corresponding
to non-hyperelliptic curves is given by the identification with a global quotient
\[
\PP \hfami_g^{non-hyp} \quad \cong \quad
\left\{
(C,H) \subset \PP^{g-1} \;\left|\; 
\begin{array}{cl}
C  \text{ a smooth canonical curve} \\
H  \text{ a hypersurface}
\end{array} 
\right.\right\}
\Big/ \raisebox{-2mm}{$\PP GL_{g}$}
\]
Then $\PP \hfami_g^{tri}$ is non-empty only if $g=3k+1$ and it is 
closed in $\PP \hfami_g^{non-hyp}$
corresponding to pairs $(C,H)$ such that \\
$C$ meets the following equivalent conditions:
\begin{enumerate}
\item
it is trigonal of Maroni invariant $k+1$.
\item
it is trigonal with $S_C$ projectively equivalent to $S_{2k,k-1}$.
\end{enumerate}
$H$ meets the following equivalent conditions:
\begin{enumerate}
\item[iii)]
it intersects $C$ in a single divisor of its
$g_3^{\,1}$.
\item[iv)]
it intersects the ruled surface $S_C$ in its section of negativ
self-intersection and a single divisor of its ruling.
\end{enumerate}

The aim of the next steps is to successively reduced the dimension
of the spaces involved in the global quotient.
For the moment, we consider the whole strictly trigonal locus at once:
\[
\PP \hfami_g^{tri} \quad = \quad 
\PP \hfami_g^{tri}(6k) \,\stackrel{\cdot}{\cup}\,
\PP \hfami_g^{tri}(4k,2k) \,\stackrel{\cdot}{\cup}\,
\PP \hfami_g^{tri}(2k,2k,2k).
\]

\begin{prop}
The moduli space $\PP \hfami_g^{tri}$, $g=3k+1$, with its orbifold structure can be 
represented as a global quotient
\[
\left\{
C \mid C \text{ canonical curve on } S_{2k,k-1}
\right\}
/ Stab_{\PP GL}(S_{2k,k-1},H_0)
\]
\end{prop}

\begin{proof}
The group of projective equivalences $\PP GL_{3k+1}$ is transitive on all pairs 
$S,H$ in $\PP^{3k}$
such that $S$ is projectively equivalent to $S_{2k,k-1}$ and
such that  $H$ has property $iv)$ above w.r.t.\ $S$.

In fact, the map $(C,H)\mapsto S_C,H$ is $\PP GL_{3k+1}$-equivariant
and induces an isomorphism in the orbifold sense of the two quotients.
\qed
\end{proof}

Here we see the first instance of how to pass from a quotient description 
by the action of a group $G$ on a space to that of a smaller group on a smaller
space:
\begin{quote}
Take the fibre of a $G$-equivariant map to a $G$-homogenous space and
the stabilizer subgroup.
\end{quote}
This kind of passing between quotient description is an elementary kind of
Morita equivalence.

\begin{prop}
The moduli space $\PP \hfami_g^{tri}$, $g=3k+1$, with its orbifold structure can be 
represented as a global quotient
\[
\left\{
f\in V^k \mid f=0 \text{ non-singular on } \FF_{k+1}
\right\}
/ G\times \CC^*
\]
where $G$ is the group of transformations induced by 
\[
x,y \quad \mapsto \quad
a x + a_0, b y + b_{k+1}x^{k+1} + \dots + b_0
\]
with $a,b\in \CC^*$, $a_0,b_0,\dots,b_{k+1} \in \CC$,
and the factor $\CC^*$ acts by scalar multiplication on $V^k$.
\end{prop}

\begin{proof}
We use \cite{sv} which describes orbits of curves on $\FF_{k+1}$ by the action 
on a fixed dense open $\CC^2\subset \FF_{k+1}$ by a group $\tilde G$ of 
birational transformations
\[
(x,y) \quad \mapsto \quad 
\left( \frac{ax+a_0}{a'x+a'_0}, 
\frac{by+b_0+b_1x+ \dots + b_{k+1}x^{k+1}}{(a'x+a'_0)^{k+1}} \right)
\]
with $(\begin{smallmatrix}a & a_0\\ a' & a'_0\end{smallmatrix})$ invertible, $b\in \CC^*$
and $b_0,\dots,b_{k+1} \in\CC$.

Their result \cite[Prop.1.2]{sv} implies that the projective equivalence classes
of trigonal canonical curves of positive Maroni invariant $k+1$
correspond bijectively to $\tilde G$-orbits of smooth curves on $\FF_{k+1}$ given
by some $f\in V^k$.

Now we identify the open $\CC^2\subset \FF_{k+1}$ with the complement of 
$H_0\cap S_{2k,k-1}$. Then the stabilizer of $(S_{2k,k-1},H_0)$ is identified
with the subgroup of $\tilde G$ corresponding to biregular transformations.
They are obtained for $a'=0$ which implies $a'_0\neq 0$. We may normalize to
$a'_0=1$ and get our claim.
\qed
\end{proof}

Let us state again, that the locus $\PP \hfami_g^{tri}$, $g=3k+1$, which we just
have given as a global quotient, decomposes
into the three loci,
\[
\PP \hfami_g^{tri}(6k),\quad
\PP \hfami_g^{tri}(4k,2k), \quad
\PP \hfami_g^{tri}(2k,2k,2k)
\]
corresponding to curves which intersect $L_0$ in $1,2$ and $3$ points
respectively, which is equal to the number of zeros of 
$sy^3+ r_{k+1}y^2+p_{2k+2}y+q_{3k+3}$.

\section{transformation steps}

\newcommand{\inft}{\circ}

In this section we obtain a few preliminary results on transformations.

\begin{lemma}
\label{yrTschirn}
If $f\in V_k$ 
is regular,
then for $r_{\!\inft}\in\CC$ the transformation 
\[
x, y \quad \mapsto \quad x, y -\frac1{3s}
\left(r(x) - r_{\!\inft} x^{k+1}\right)
\]
maps $f$ to $f'$ with $r'(x)=r_{\!\inft}x^{k+1}$.

If moreover $\deg_x (r(x)-r_\inft x^{k+1}) < k$ then 
\[
p'_{2k+1}=p_{2k+1}, \quad p'_{2k+2}=p_{2k+2},
\quad q'_{3k+2}=q_{3k+2}, \quad q'_{3k+3}=q_{3k+3}.
\]
\end{lemma}

\begin{proof}
The transformation is well defined, since regularity of the curve implies
$s\neq0$. It is then simply the Tschirnhaus transformation for the polynomial in $y$
transforming the coefficients which are polynomials in $x$. 

If the given degree bound holds, it is easy to check that the polynomials $r,p,q$ are
only affected in degrees less than $k,2k+1,3k+2$ respectively.
\qed
\end{proof}

The following three results are proved by similar elementary arguments:

\begin{lemma}\label{ypTschirn}
If $f\in V_k$ with $q_{3k+3}=p_{2k+2}=0$ and $sr_{k+1}q_{3k+2}\neq 0$ 
then the transformation 
\[
x, y \quad \mapsto \quad x , 
y-\frac1{2}\,\frac{p_{2k+1}}{r_{k+1}}x^k
\]
maps $f$ to $f'$ with $q'_{3k+3}=p'_{2k+2}=0$, $s'r'_{k+1}q'_{3k+2}\neq 0$ and
\begin{eqnarray*}
p'_{2k+1} & = & 0.
\end{eqnarray*}
\end{lemma}

\begin{lemma}
\label{xrTschirn}
If $f\in V_k$ with $q_{3k+3}=p_{2k+2}=p_{2k+1}=0$ and $sr_{k+1}q_{3k+2}\neq 0$ 
then the transformation 
\[
x, y \quad \mapsto \quad x -\frac1{k}\,\frac{r_k}{r_{k+1}},y
\]
maps $f$ to $f'$ with $q'_{3k+3}=p'_{2k+2}=p'_{2k+1}=0$, $s'r'_{k+1}q'_{3k+2}\neq 0$ and
\begin{eqnarray*}
r'_k & = & 0.
\end{eqnarray*}
\end{lemma}

\begin{lemma}
\label{xqTschirn}
If $f\in V_k$ with $q_{3k+3}=p_{2k+2}=0$ and $sq_{3k+2}\neq 0$ 
then the transformation 
\[
x, y \quad \mapsto \quad x -\frac1{3k+2}\,\frac{q_{3k+1}}{q_{3k+2}}, y
\]
maps $f$ to $f'$ with $q'_{3k+3}=p'_{2k+2}=0$, $s'q'_{3k+2}\neq 0$ and
\begin{eqnarray*}
q'_{3k+1} & = & 0, \qquad p'_{2k+1} \quad = \quad p_{2k+1}.
\end{eqnarray*}
\end{lemma}

Also the next result is similar, though it needs some more care to show its
validity.

\begin{lemma}
\label{L0Tschirn}
If $f\in V_k$ with $r(x)=0$ and $\{f=0\}$ non-singular intersecting the
line $L_0$ at infinity in two points,
then the transformation 
\[
x, y \quad \mapsto \quad x, y -\frac{3q_{3k+3}}{2p_{2k+2}}x^{k+1}
\]
maps $f$ to $f'$ with $q'_{3k+3} =0=p'_{2k+2}$ and
\begin{eqnarray*}
r'(x) & = & r'_{k+1}x^{k+1},\quad r'_{k+1}\neq 0
\end{eqnarray*}
\end{lemma}

\begin{proof}
The curve intersects the line $L_0$ in the roots of 
\[
sy^3+p_{2k+2}y+q_{3k+3}
\]
thus $p_{2k+2}\neq0$, for otherwise the number of roots is either
one or three according to $q_{3k+3}$ equal to zero or not.
Hence the transformation is well defined. 

The vanishing discriminant implies $27q_{3k+3}^2s=-4p_{2k+2}^3$, thus $q_{3k+3}\neq0$ and one
can check that the polynomial
\[
-4p_{2k+2}^3y^3 + 27q_{3k+3}^2p_{2k+2}y+ 27q_{3k+3}^3
\]
has a simple root at $3q_{3k+3}/p_{2k+2}$ and a double root at $-3q_{3k+3}/2p_{2k+2}$.
Hence the given transformation yields a polynomial $f'$ with 
$r'_{k+1}=9sq_{3k+3}/2p_{2k+2}\neq0$ and a double root at $0$, so finally
$p'_{2k+2}=0=q'_{3k+3}$.
\qed
\end{proof}

We further need a non-vanishing result to be exploited later.

\begin{lemma}
\label{qnon-van}
Suppose $f\in V_k$ restricted to the line $L_0$ at infinity 
has a double or triple zero at $y=0$
and $\{f=0\}$ is non singular on $\FF_{k+1}$, then 
\[
q_{3k+2} \quad \neq \quad 0
\]
for the coefficient of $x^{3k+2}$ in the monomial expansion of $f$.
\end{lemma}

\begin{proof}
At $y=0$ on the line $L_0$, the linear local expansion of $f$ in $x_0=1/x$ and $y$ is
\[
f(x_0,y) \quad = \quad q_{3k+3} + p_{2k+2}y + q_{3k+2}x_0 + h.o.t.
\]
The first two coefficients are zero, since $x=\infty, y=0$ is a zero of $f$
respectively at least doubly so along $x=\infty$. In order to have non-vanishing
gradient at $x=\infty, y=0$ the last coefficient $q_{3k+2}$ must be non-vanishing.
\qed
\end{proof}

A final prerequisite concerns an alternative choice of basis for the $(\CC^*)^3$
subgroup of transformations.

\begin{lemma}
\label{cstarmat}
The actions of $(\CC^*)^3$ on $V^k$ given by respectively
\[
x^iy^j \quad \mapsto\quad a^ib^jc x^iy^j, \qquad
x^iy^j \quad \mapsto\quad \lambda^{6k+4-2i-2kj-j}\mu^{3k+3-i-kj-j}\varrho^{9k+6-3i-3kj-2j} x^iy^j
\]
are equal under a group automorphism.
\end{lemma}

\begin{proof}
The group automorphism is given by
\[
a\mapsto\lambda^{-2i}\mu^{-i}\varrho^{-3i}
b\mapsto\lambda^{-2kj-j}\mu^{-kj-j}\varrho^{-3kj-2j}
c\mapsto\lambda^{6k+4}\mu^{3k+3}\varrho^{9k+6}.
\]
It is obviously a group homomorphism, it has an inverse, since the corresponding 3 by 3 matrix
is invertible over the integers. Finally it translates the first action to the second one.
\qed
\end{proof}

\begin{prop}
\labell{trans1}
If $f\in V^k$ and $\{f=0\}$ non-singular intersecting the
line $L_0$ at infinity in one point,
then in the $G$-orbit of $f$ intersects $V_6$ in the orbit of some
\[
f' \in V^k \quad \text{with} \quad
s=q_{3k+2}=1, r_0=\dots=r_{k+1}=p_{2k+2}=q_{3k+1}=q_{3k+3}=0
\]
by the residual action of $\CC^*$ via
\[
p_i \mapsto t^{6k+4-3i} p_i, \quad
q_i \mapsto t^{9k+6-3i} q_i
\]
\end{prop}

\begin{proof}
Since $\{f=0\}$ is non-singular,
by Lemma \ref{yrTschirn} with $r_{\!\inft}=0$ any $f$ of the claim transforms to a 
polynomial $f'$ with $r'(x)=0$. By assumption there is only one zero on $L_0$ for $f$, 
thus $y=0$ is a triple zero on $L_0$ for $f'$, $p'_{2k+2}=0$ and $q'_{3k+3}=0$.
Lemma \ref{qnon-van} then implies $q'_{3k+2}\neq0$.
By Lemma \ref{xqTschirn} we may then additionally impose $f'$ to have $q'_{3k+1}=0$.

The stabilizer group of the locus of such $f'$ is the abelian group $(\CC^*)^3$ of Lemma
\ref{cstarmat}. The action by $\lambda,\mu,\varrho$ on the coefficients of $f'$ is
\[
s \mapsto \lambda  s, \quad
p_i \mapsto \lambda^{4k+3-2i} \mu^{2k+2-i} \varrho^{6k+4-3i} p_i, \quad
q_i \mapsto \lambda^{6k+4-2i} \mu^{3k+3-i} \varrho^{9k+6-3i} q_i
\]
With $\lambda=1/s$ and $\mu=1/q_{3k+2}$ we arrive at $f'$ in the claimed set with
the residual action of $t=\varrho$ via $p_i \mapsto \varrho^{6k+4-3i} p_i, 
q_i \mapsto \varrho^{9k+6-3i} q_i$ as given in the claim.
\qed
\end{proof}

\begin{prop}
\labell{trans2}
If $f\in V^k$ and $\{f=0\}$ non-singular intersecting the
line $L_0$ at infinity in two points,
then in the $G$-orbit of $f$ there is a unique $f'\in V^k$ with
\[
s'=r'_{k+1}=q'_{3k+2}=1, \quad 
r'_0=\dots=r'_k=0, \quad p'_{2k+1}=p'_{2k+2}=q'_{3k+3}=0.
\]
\end{prop}

\begin{proof}
Similar to the last proof,
a succession of transformations according to Lemma \ref{yrTschirn} with $r_{\!\inft}=0$,
Lemma \ref{L0Tschirn}, Lemma \ref{ypTschirn}, Lemma \ref{xrTschirn} and 
Lemma \ref{yrTschirn} again but with $r_{\!\inft}=r_{k+1}$ is possible thanks to 
Lemma \ref{qnon-van}.
It yields $f'$ in the orbit of $f$ with 
\[
s'r'_{k+1}q'_{3k+2}\neq 0 \quad \text{and} \quad
r_0=\dots=r_k=0, \quad p_{2k+1}=p_{2k+2}=0.
\]
The stabilizer group of the locus of such $f'$ is again the abelian group $(\CC^*)^3$ of Lemma
\ref{cstarmat}. The action by $\lambda,\mu,\varrho$ on non-vanishing coefficients of $f'$ is
\[
s \mapsto \lambda  s, \quad
r_{k+1} \mapsto \varrho^{-1} r_{k+1}, \quad
q_{3k+2} \mapsto \mu q_{3k+2}
\]
With $\lambda=1/s$, $\mu=1/q_{3k+2}$ and $\varrho=r_{k+1}$ we arrive 
at $f'$ in the claimed set with no residual action.
\qed
\end{proof}

\begin{prop}
\labell{trans3}
If $f\in V^k$ and $\{f=0\}$ non-singular intersecting the
line $L_0$ at infinity in three points,
then the $G$-orbit of $f$ intersects $V_{2,2,2}$ in the orbit of some
\[
f' \in V^k \quad \text{with} \quad
s=1, r_0=\dots=r_{k+1}=0
\]
by the residual action of $(a,a_0,t) \in (\CC^*\rtimes\CC) \times \CC^*$ via
\[
p(x) \mapsto t^2 p(ax+a_0), \quad
q(x) \mapsto t^3 q(ax+a_0).
\]
\end{prop}

\begin{proof}
By a transformation according to Lemma \ref{yrTschirn} with $r_{\!\inft}=0$ the
polynomial $f'$ is obtained. The stabilizing subgroup can be identified with
transformations induced by 
\[
x \mapsto a x + a_0
\]
and the action of a scalars $\lambda,t\in \CC^*$ via
\[
s \mapsto \lambda  s, \quad
p_i \mapsto \lambda t^2 p_i, \quad
q_i \mapsto \lambda t^3 q_i
\]
With $\lambda=1/s$ we arrive at the form of $f'$ of the claim and the residual 
action there.
\qed
\end{proof}

Let us note, that the three propositions make statements about the slices
of $V^k$ given by the affine subspaces
\begin{eqnarray*}
V_1^k & = & 
\{\, f \in V^k \mid   s=q_{3k+2}=1 , r=0, p_{2k+2}=0, q_{3k+3}=0\}
\\
V_{2}^k & = & 
\{\, f \in V^k \mid   s=r_{k+1}=q_{3k+2}=1 , r_0=\dots=r_k=
p_{2k+1}=p_{2k+2}= q_{3k+3}=0\}
\\
V_{3}^k & = & 
\{\, f \in V^k \mid   s=1, r_0=\dots=r_{k+1}=0 \}
\end{eqnarray*}

\section{proof of the main theorem}

Next we apply the results of the previous section to a proof
of our main theorem.
\medskip

\begin{proof}[ of Theorem \ref{orbi}]
Let us start to address the first claim about the global quotient for 
$\PP\hfami^{tri}_{3k+1}(6k)$.

At the end of section 2 we established the identification with the
global quotient
\[
\left\{
f\in V^k \text{ regular} \mid C_f \text{ on } \FF_{k+1} \text{ intersects $L_0$ in one point}
\right\}
/ G\times \CC^*.
\]
In particular, the hypotheses of Prop.\ref{trans1} are met by all the $f$ involved.

The transformations used in the proof of Prop.\ref{trans1} can be performed 
simultaneously on all elements $f$, since they depend algebraically on the
coefficients of $f$.

Hence we get the global quotient description of $\PP\hfami^{tri}_{3k+1}(6k)$ by
\[
\left\{
f\in V^k \text{ regular} \mid 
s=q_{3k+2}=1, r(x)=0, p_{2k+2}=q_{3k+1}=q_{3k+3}=0
\right\}
/ \CC^*
\]
Define the linear subsystem $\lfami_1$ of $|3\sigma_0|$ by
\[
\lfami_1 \quad = \quad 
\{ [f] \in\PP V^k  \mid  s=q_{3k+2},\, r(x)=0,\, p_{2k+2}=q_{3k+1}=q_{3k+3}=0\,\}.
\]
Then the affine subspace
\[
\{ f \in V^k  \mid  s=q_{3k+2}=1,\, r(x)=0,\, p_{2k+2}=q_{3k+1}=q_{3k+3}=0\,\}
\]
maps bijectively to the complement of the hyperplane $s=0$ in $\lfami_1$.
Since this hyperplane is contained in the discriminant $\dfami_1$ of $\lfami_1$,
that map induces an isomorphism
\[
\{ f \in V^k \text{ regular}  \mid  s=q_{3k+2}=1,\, r(x)=0,\, p_{2k+2}=q_{3k+1}=q_{3k+3}=0\,\}
\quad\cong\quad
\lfami_1-\dfami_1.
\]
We transfer the $\CC^*$-action on the left hand side through this isomorphism to
the right hand side to get the claimed global quotient for $\PP\hfami^{tri}_{3k+1}(6k)$:
\[
(\lfami_1-\dfami) / \CC^*.
\]

The analogous strategy yields $\PP\hfami^{tri}_{3k+1}(4k,2k)\cong(\lfami_2-\dfami)$ in the second case with
\[
\lfami_2 \quad = \quad 
\{ [f] \in\PP V^k  \mid  s=q_{3k+2}=r_{k+1},\, r_0=0,\dots,r_k=0,\, p_{2k+1}=p_{2k+2}=q_{3k+3}=0\,\}
\]
and with $\lfami_3 = \{ [f] \in\PP V^k  \mid  r(x)=0\,\}$ yields
\[
\PP \hfami^{tri}_{3k+1} \quad \cong \quad 
(\lfami^k_{3} -\dfami_3 ) / \CC\rtimes (\CC^*)^2.
\]
To get the third case, it suffices then to additionally discard the divisor $\dfami'$ of pairs belonging
to $\PP\hfami^{tri}_{3k+1}(4k,2k)\stackrel{\cdot}{\cup}\PP\hfami^{tri}_{3k+1}(6k)$.
\qed
\end{proof}

\section{orbifold fundamental groups and genus $4$}

In this final section we address the proofs of the corollaries from the introduction
concerning the orbifold fundamental group of trigonal loci and the special situation
of genus $4$ curves.
\medskip

The orbifold fundamental group for an orbifold given as a global quotient of
a manifold $X$ by a topological group $G$ is by definition the topological
fundamental group of the quotient
\[
(EG \times X) / G
\]
where $G$ acts freely on the contractible space $EG$ and 
diagonally on the product.

Then there is a long exact homotopy sequence associated to the diagonal free action
\[
\pi_2 \big((EG \times X) / G\big) \to \pi_1 G \to \pi_1 (EG \times X) \to
\pi_1 \big((EG \times X) / G\big) \to \pi_0 G
\]
in which we may identify $\pi_1 (EG \times X)=\pi_1 X$ and $\pi_1 \big((EG \times X) / G\big)
=\pi_1^{orb}(X / G)$.

\medskip

\begin{proof}[ of Corollary 1.3]
The moduli space $\PP \hfami^{tri}_{3k+1}(6k)$ was shown to be representable as a quotient of the 
complement $U_1^k:=V^k_1-\dfami$ of the discriminant in $V_1^k$ 
by the group $\CC^*$.
Hence we get from the long exact homotopy sequence above
\[
\ZZ \quad\to\quad \pi_1 U_1^k \quad\to\quad
\pi_1^{orb} (U_1^k / \CC^*) \quad\to\quad 1
\]
The map from $\ZZ=\pi_1\CC^*$ is injective, since so is the map to $H_1 U_1^k$, which is
essentially multiplication by the non-torsion homology class 
represented by any free orbit of $S^1\subset \CC^*$.

The affine space $V_1^k$ is an unfolding of the singularity given by the polynomial $y^3 + x^{3k+2}$,
namely the unfolding over the vector space spanned by the monomials $x^iy^j$ of weighted degree
$3i+(3k+2)j<9k+6$.
By the arguments given in \cite[section 3.2]{mir} 
the fundamental group of the discriminant complement $U_1^k$ 
is naturally isomorphic to that of the discriminant complement of the universal
unfolding of $y^3 + x^{3k+2}$. Thus by definition, it is the discriminant knot group 
$\pi^K(y^3 + x^{3k+2})$.

Hence we get the short exact sequence of the claim.
\qed 
\end{proof}
\medskip

\begin{remark}
A finite presentation for $\pi^K(y^3+x^{3k+2})$ was given in \cite{intern}
and with the argument in \cite{mir} it can also be given with
generators $t_i$ $i=1,..., 6k+2$
and relations
\begin{enumerate}
\item
of braid type : $t_i t_j t_i = t_j t_i t_j$ for all $i,j$ with $|i-j|\leq 2$
\item
of commutation type : $t_i t_j = t_j t_i$ for all $i,j$ with $|i-j|> 2$
\item
of triangle type: $t_i t_{i+1} t_{i+2} t_i = t_{i+1} t_{i+2} t_i t_{i+1}$ for all $i<6k+1$
\end{enumerate}
The central element which has to be factored can be expressed as
\[
(t_1 \dots t_{6k+2})^{9k+6}
\]
\end{remark}

\begin{proof}[ of Corollary 1.4]
The moduli space $\PP \hfami^{tri}_{3k+1}(2k,2k,2k)$ was shown to be representable as a 
quotient of the complement $U_{3}^k:=V^k_3-\dfami-\dfami'$ of $\dfami\cup\dfami'$ in $V_{3}^k$ 
by the group $\CC \rtimes (\CC^*)^2$.
Hence we get from the long exact homotopy sequence above
\[
\ZZ^2 \quad\to\quad \pi_1 U_3^k \quad\to\quad
\pi_1^{orb} U_3^k / \CC^* \quad\to\quad 1
\]

The map from $\ZZ^2=\pi_1(\CC^*)^2$ is injective, since so is the map to $H_1 U^k_3$:
Indeed, linking with the divisors $\dfami$ and $\dfami'$ gives a map $H_1 U^k_3\to\ZZ^2$
such that the composition  $\ZZ^2=\pi_1(\CC^*)^2\to H_1 U^k_3\to\ZZ^2$ is injective.

Moreover, we need to identify 
\[
\pi_1 U_{3}^k  \quad \cong \quad
\pi^K(y^3+x^{3k+3}) \rtimes \pi^K(y^3+x^2)
\]
This is possible since the map 
\begin{eqnarray*}
U_{3}^k  & \to & \{ Y^3 + p Y + q \mid 4p^3\neq -27q^2 \}
\\
f & \mapsto & Y^3 + p_{2k+2}Y + q_{3k+3}
\end{eqnarray*}
is close enough to a fibre bundle.
In fact, by the results in \cite{bries}, each point in the base has a topological
disc neighbourhood, such that the preimage has fundamental group
$\pi^K(y^3+x^{3k+3})$, the discriminant knot group of the fibre over 
$Y^3+1$.

The base space is a versal unfolding of the polynomial $Y^3$ which
defines a simple singularity.
In particular, it has fundamental group the discriminant knot group $\pi^K(y^3+x^2)$.
Since its universal covering space is contractible the second homotopy group
is trivial.

By the above, there is a short exact sequence coming from a longer exact homotopy
sequence:
\[
1 \quad \to \quad
\pi^K(y^3+x^{3k+3}) \quad \to \quad
\pi_1 U_{3}^k \quad \to \quad
\pi^K(y^3+x^2)
\quad \to \quad 1
\]
There exists a section for the topological spaces given by the map 
\[
Y^3 + p Y + q \quad \mapsto \quad y^3 + px^{2k+2}y + q x^{3k+3} + 1
\]
accordingly the group $\pi_1 U_{3}^k$ above in the middle is a semi-direct product of the
other two as claimed.
\qed
\end{proof}

\begin{remark}
A finite presentation for the two factors is again known, but how the quotient acts
on the normal factor is not known.
\end{remark}
\medskip

For the proof of the final corollary \ref{pezzo} we note that $\CC^8-\dfami_{E_8}$ is 
$\CC^*$-equivariantly equal to $V_1^1-\dfami$ since
\[
V_1^1 \quad = \quad y^3+x^5 + \span \{1,x,x^2,x^3,y,yx,yx^2,yx^3\}
\]
In particular, both have fundamental group $\pi^K(y^3+x^5)=Ar(E_8)$ and the 
generator of $\pi_1(\CC^*)$ maps to the generator of the center.
By the proof of the theorem $\PP \hfami^{tri}_4(6)\cong(V_1^1-\dfami)/\CC^*$,
so the corollary follows as soon as the following lemma is proven.

\begin{lemma}
\[
\hsix \quad = \quad \hfami^{tri}_{4}(6)
\]
\end{lemma}

\proof
Let us first say something about even and odd which was already mentioned in the introduction:
For pairs $(C,D)$ in the strictly trigonal loci
the canonical divisors $D=2kL\big|_C$ come with a theta-characteristic
$kL\big|_C$ of parity equal to that of
\[
h^0(kL\big|_C)\quad=\quad h^0(\ofami_C(kL))\quad=\quad k+1 \quad\cong_2\quad g \mod 2.
\]
Indeed $h^0(kL)=k+1 + h^1(\ofami_\FF(kL-C))$ from 
the long exact cohomology sequence of
\[
\begin{array}{cccccccccc}
0 & \to & \ofami_\FF(kL-C) & \to & \ofami_\FF(kL)  & \to & \ofami_C(kL) & \to & 0
\end{array}
\]
and $h^1(\ofami_\FF(kL-C))$ vanishes by 
the long exact cohomology sequence of
\[
\begin{array}{cccccccccc}
0 & \to & \ofami_\FF(-E) & \to & \ofami_\FF  & \to & \ofami_E & \to & 0
\end{array}
\]
where $E=2\sigma_0+\sigma_\infty+L\in |C-kL|$ is effective and connected.
\medskip

In the case at hand $g=4$ therefore every pair $C,D$ on the right determines
an even theta-characteristic and thus also belongs to the left hand side.
Conversely every pair $C,D$ in the set on the left
has a curve $C$ with an effective theta characteristic. Hence $C$ belongs to the thetanull
divisor of genus $4$ curves. $C$ is not hyperelliptic, since every hyperelliptic curve with
a $6$-uple canonical divisor is in the disjoint stratum $\hfami^{hyp}_4(6)$.
Therefore $C$ maps canonically to a trigonal curve on the quadric cone
and is a trigonal curve with a unique $g_3^{\,1}$. 

Consider now the effective canonical divisor $D$: By assumption it is the double
$D=2D'$ of an effective divisor $D'$ which is an even theta-characteristic.
Therefore $D'$ has degree $3$ and being effective and even implies that
$D'$ belongs to the unique $g_3^{\,1}$ of $C$.

We may thus conclude, that $(C,D)$ belongs to the right hand side as well.
\qed
\endproof

Similar claims are not true for larger $k$ simply for dimension reasons: 
The dimension of $\hfami(2g-2)$ is $2g-1=6k+1$ and so is the dimension of its
image in the moduli space. On the other hand, the dimension of the locus
of trigonal curves with maximal Maroni invariant in genus $g=3k+1$ is
$5k+3$ and those with a total ramification point form a divisor in it of
dimension $5k+2$.
\bigskip

\paragraph{\bf Acknowledgements}
The author wants to acknowledge the ICERM at Brown University for the invitation to
the conference on Braids in Symplectic and Algebraic Geometry
and its stimulating atmosphere. 
He thanks Fabrizio Catanese, Stephen Coughlan and Eduard Looijenga 
for interesting and clarifying discussions.


\end{document}